\documentclass[12pt]{article}
\usepackage{amssymb}
\usepackage{amsfonts}
\usepackage{amsmath}
\usepackage[usenames]{color}
\usepackage{mathrsfs}
\usepackage{amsfonts}
\usepackage{amssymb,amsmath}
\usepackage{CJK}
\usepackage{cite}
\usepackage{cases}
\usepackage{amsthm}
\usepackage{amssymb}
\usepackage{graphicx}
\usepackage{amsmath}
\usepackage[all]{xy}
\usepackage{enumerate}
\usepackage{amscd}
\usepackage{cases}

\def\cl{\centerline}

\def\G{\mathcal{G}}

\def\vs{\vspace*}
\def\S{{\mathfrak{S}}}

\def\V{\mathcal{V}}

\def\Z{\mathbb{Z}}
\def\N{\mathbb{N}}
\def\H{\mathcal{H}}
\def\G{\mathcal{G}}

\def\C{\mathbb{C}}

\def\ni{\noindent}

\numberwithin{equation}{section}
\newtheorem{theo}{Theorem}[section]
\newtheorem{defi}[theo]{Definition}

\newtheorem{lemm}[theo]{Lemma}
\newtheorem{prop}[theo]{Proposition}

\newtheorem{rema}[theo]{Remark}

\begin{document}
\begin{center}
\cl{\large\bf \vs{8pt}The  Heisenberg-Virasoro Lie conformal superalgebra}
\cl{ Haibo Chen$^*$, Xiansheng Dai and Yanyong Hong}
\end{center}
\footnote {
*Corresponding author: H. Chen (Hypo1025@163.com).
}
{\small
\parskip .005 truein
\baselineskip 3pt \lineskip 3pt

\noindent{{\bf Abstract:}
In this  paper, we  introduce a finite Lie conformal  superalgebra called the Heisenberg-Virasoro Lie conformal  superalgebra  $\mathfrak{s}$ by using a class of    Heisenberg-Virasoro Lie conformal
 modules.  The super Heisenberg-Virasoro algebra  of  Ramond type $\S$ is    defined by  the  formal distribution Lie superalgebra  of  $\mathfrak{s}$.
 Then we    construct  a class of simple  $\S$-modules, which are induced  from    simple modules of some    finite dimensional solvable Lie superalgebras.
 These modules are isomorphic to     simple restricted  $\S$-modules, and include the highest weight modules, Whittaker modules and high order Whittaker  modules.
As a byproduct, we  present a   subalgebra  of $\S$, which is isomorphic to the super Heisenberg-Virasoro algebra  of  Neveu-Schwarz type.
\vs{5pt}

\ni{\bf Key words:}
  Lie conformal superalgebra,  super Heisenberg-Virasoro algebra,  simple module.}

\ni{\it Mathematics Subject Classification (2020):} 17B10, 17B65, 17B68.}
\parskip .001 truein\baselineskip 6pt \lineskip 6pt
\tableofcontents
\section{Introduction}
Throughout the present paper,  we denote by $\C,\C^*,\Z,\Z_+$ and $\N$ the sets of complex numbers, nonzero complex numbers,  integers,
 nonnegative integers and positive integers, respectively.
   All vector superspaces
(resp. superalgebras, supermodules)   and
 spaces (resp. algebras, modules)    are considered to be over
$\C$.

The  concept of  Lie conformal  (super)algebras was  introduced  by Kac,
which encodes an axiomatic description of the
operator product expansion (or rather its Fourier transform) of chiral fields in conformal field theory (see \cite{K1,K3}).
The theory of Lie conformal  (super)algebras gives us a powerful tool
for the study of infinite dimensional Lie (super)algebras satisfying
the locality property   (see \cite{K2}).    Lie (super)algebras and Lie conformal (super)algebras have been studied extensively by researchers in recent years
 (see,  e.g. \cite{CHSX,CHS0,X,SXY}).

As an important infinite dimensional Lie  algebra, \textit{the twisted Heisenberg-Virasoro algebra} $\H$    is   the universal central extension of
the Lie algebra $$\overline {\mathcal{H}}:=\left\{f(t)\frac{d}{dt}+g(t)\,\Big|\, f(t),g(t)\in\C[t,t^{-1}]\right\}$$ of differential operators of order at
most one,  which was studied   in  \cite{ADK}.
They established a connection  between the second cohomology
of certain moduli spaces of curves and the second cohomology of the Lie algebra of differential
operators of order at most one.

Highest weight modules and  Whittaker modules are  two classes of important modules in representation theory. In \cite{ADK},  the authors precisely determined the determinant formula of the Shapovalov form for the Verma modules, and   showed   any simple
highest weight module for $\H$ is isomorphic to the tensor product of a simple module for the Virasoro
algebra and a simple module for the infinite dimensional Heisenberg algebra. In  \cite{B}, Billig obtained the character formula for simple highest weight
modules under some conditions of trivial actions.     The simple weight modules with
finite dimensional weight spaces (namely, Harish-Chandra modules) over $\H$ were completely classified in \cite{LZ1}, which turn out to be modules of intermediate series or
highest (lowest) weight modules. Subsequently, a simpler and more conceptual proof of
the classification of irreducible Harish-Chandra modules  over $\H$  were   presented  in \cite{LG}.

  Whittaker modules were first   introduced  for $\mathfrak{sl}(2)$ by Arnal and Pinczon    in \cite{AP}.
 Meanwhile,
 Whittaker modules of finite dimensional complex semisimple Lie algebras were
defined in \cite{K}. From then on, they have been studied on different kinds of subjects,
especially for affine Lie algebra, infinite dimensional Lie algebra and Lie superalgebra  (see,  e.g. \cite{BM,CJ,ALZ,C,BCW,LPX0,LPX,LWZ,CG,LZ3}).
Furthermore,    Whittaker modules have also been  investigated in the framework of vertex operator algebra  theory (see \cite{ALZ,T,HY,ALPY}).

Recently, the representation theory of  super Virasoro algebra was widely researched (see e.g.  \cite{LPX,YYX1,CLL}),
which inspires us to study the super case  of twisted Heisenberg-Virasoro algebra.
Now we  present the definition of   \textit{super Heisenberg-Virasoro algebra  of
Ramond type} $\S$. It is an  infinite dimensional  Lie superalgebra
$$\mathfrak{S}=\bigoplus_{m\in\Z}\C L_m \oplus \bigoplus_{m\in\Z}\C G_{m} \oplus\bigoplus_{m\in\Z}\C I_m,$$
which  satisfies  the following  super-brackets
\begin{eqnarray*} \label{def1.1}
&&[L_{m},L_{n}]=(n-m)L_{m+n},
\
  [L_{m},G_{n}]=nG_{m+n},
   \\&& [L_{m},I_{n}]=nI_{m+n},\
 [G_{m},G_{n}]=2I_{m+n},
 \
 [I_{m},I_{n}]=[I_{m},G_{n}]=0
\end{eqnarray*} for $m,n\in\Z$.
By its definition, we have the following decomposition:
$$\S=\S_{\bar 0}\oplus\S_{\bar 1},$$
where $\S_{\bar 0}=\mathrm{span}_{\C}\{L_m,I_m\mid m\in\Z\}$, $\S_{\bar 1}=\mathrm{span}_{\C}\{G_m\mid m\in\Z\}$.
Notice that the even part  $\S_{\bar 0}$   is isomorphic to
    twisted Heisenberg-Virasoro algebra with some   trivial   center elements.
The super Heisenberg-Virasoro algebra  of
Ramond type has a $\Z$-grading by the eigenvalues of the adjoint action of $L_0$. Then $\S$ possesses the following triangular decomposition:
$$\S=\S_+\oplus\S_0\oplus\S_-,$$
where $\S_{+}=\mathrm{span}_{\C}\{L_m,I_m,G_m\mid m\in\N\}$, $\S_{-}=\mathrm{span}_{\C}\{L_m,I_m,G_m\mid -m\in\N\}$, $\S_{0}=\mathrm{span}_{\C}\{L_0,I_0,G_0\}$. Note that $\C I_0$ is the center of $\S$.

The actions of elements in the positive part of the algebra are locally finite, which is the same property of highest weight modules and Whittaker modules (see, e.g. \cite{MZ1,OW}).
 Mazorchuk and Zhao in \cite{MZ} proposed   a very general construction of simple Virasoro modules which  generalizes
and includes both highest weight modules and   various versions of Whittaker modules.
This construction
enabled them to classify all simple Virasoro modules that are locally finite over a positive
part.
Inspired by this, new simple modules over the Virasoro algebra and its some (super)extended cases have been investigated (see \cite{CG,G,CHS,LPX,CGHS,LZ2,MW}).
The high order Whittaker modules  were studied in \cite{LGZ} by generalized some known results on Whittaker modules of the Virasoro algebra.  They  obtained concrete bases for all irreducible Whittaker modules (instead of a quotient of modules). Observe  that     highest weight modules,   Whittaker modules and   high order
 Whittaker modules over $\S$  are restricted  modules. This makes us  to study restricted    modules for the the super Heisenberg-Virasoro algebra of Ramond type.

The rest of this paper is organized as follows.
In Section $2$, we   recall some notations and definitions of Lie  conformal superalgebra and Lie superalgebra.
In Section $3$, a   super  Heisenberg-Virasoro algebra  is introduced from the view of the formal distribution Lie superalgebra  of the Heisenberg-Virasoro Lie conformal superalgebra.
In Section $4$, a class of  simple  restricted $\S$-modules  are  constructed, which  generalize   and include  the highest weight modules, Whittaker modules and high order Whittaker modules.
In Section $5$, we give  a characterization of simple restricted modules for the super Heisenberg-Virasoro algebra   of  Ramond type. More precisely, it reduces the problem of classification of simple restricted $\S$-modules to classification of simple modules over some finite dimensional solvable Lie superalgebras.
At last, we show some examples of restricted $\S$-modules, such as the highest weight modules, Whittaker
modules and high order Whittaker modules  and so on.

\section{Preliminaries}
In this section, we will   make some  preparations for   later use.
\subsection{Lie conformal superalgebra}
Some definitions and results related to Lie conformal superalgebras and conformal
modules    are recalled (see \cite{DK,K1,K3}).

A Lie conformal superalgebra is called \textit{finite} if it is finitely generated as a $\C[\partial]$-module, or else it is called \textit{infinite}.
\begin{defi}\rm \label{D1}
{\rm (1)} \textit{A   Lie  conformal  superalgebra}  $S=S_{\bar0}\oplus S_{\bar1}$ is a $\Z_2$-graded $\C[\partial]$-module  endowed with a
  $\lambda$-bracket $[a{}\, _\lambda \, b]$
which defines a
linear map $S_{\alpha}\otimes S_{\beta}\rightarrow \C[\lambda]\otimes S_{\alpha+\beta}$, where $\alpha,\beta\in\Z_2$ and $\lambda$ is an indeterminate, and satisfies the following axioms:
\begin{equation*}
\aligned
&[\partial a\,{}_\lambda \,b]=-\lambda[a\,{}_\lambda\, b],\
[a\,{}_\lambda \,\partial b]=(\partial+\lambda)[a\,{}_\lambda\, b],\\
&[a\, {}_\lambda\, b]=-(-1)^{|a||b|}[b\,{}_{-\lambda-\partial}\,a],\\
&[a\,{}_\lambda\,[b\,{}_\mu\, c]]=[[a\,{}_\lambda\, b]\,{}_{\lambda+\mu}\, c]+(-1)^{|a||b|}[b\,{}_\mu\,[a\,{}_\lambda \,c]]
\endaligned
\end{equation*}
for $a\in S_\alpha,b\in S_\beta,c\in S_\gamma$ and $\alpha,\beta,\gamma\in\Z_2$.

{\rm (2)}
\textit{A  conformal module} $M=M_{\bar0}\oplus M_{\bar1}$ over a Lie conformal
superalgebra $S$ is a  $\Z_2$-graded $\C[\partial]$-module endowed with a $\lambda$-action $S_\alpha\otimes M_\beta\rightarrow \C[\lambda]\otimes M_{\alpha+\beta}$ such that
\begin{eqnarray*}
&&(\partial a)\,{}_\lambda\, v=-\lambda a\,{}_\lambda\, v,\ a{}\,{}_\lambda\, (\partial v)=(\partial+\lambda)a\,{}_\lambda\, v,
\\&&
a\,{}_\lambda\, (b{}\,_\mu\, v)-(-1)^{|a||b|}b\,{}_\mu\,(a\,{}_\lambda\, v)=[a\,{}_\lambda\, b]\,{}_{\lambda+\mu}\, v
\end{eqnarray*}
for all $a\in S_\alpha,b\in S_\beta,v\in M_\gamma$ and $\alpha,\beta,\gamma\in\Z_2$.
\end{defi}

In particular, there is an important Lie superalgebra associated with  the Lie conformal superalgebra.
Let $S$ be a Lie conformal superalgebra.
 Assume that  Lie$(S)$ is the quotient
of the vector space with basis $a_n$ $(a\in S, n\in\mathbb{Z})$ by
the subspace spanned over $\mathbb{C}$ by
elements:
$$(\alpha a)_n-\alpha a_n,~~(a+b)_n-a_n-b_n,~~(\partial
a)_n+na_{n-1},~~~\text{where}~~a, b\in S, \alpha\in \mathbb{C}, n\in
\mathbb{Z}.$$ Note that if $a\in S_{\alpha}$ then $a_n\in \mathrm{Lie}(S)_{\alpha}$ for $n\in\Z,\alpha\in\Z_2$. The operation on Lie$(S)$ is given as follows:
\begin{equation}\label{2011}
[a_m, b_n]=\sum_{j\in \mathbb{N}}\left(\begin{array}{ccc}
m\\j\end{array}\right)(a_{(j)}b)_{m+n-j},\end{equation}
 where $a_{(j)}b$ is called the \textit{$j$-th product}, given by  $[a_\lambda b]=\sum_{n=0}^\infty\frac{\lambda^{n}}{n!}(a_{(n)}b)$.
 Then
Lie$(S)$ is a Lie superalgebra and it is called  \textit{a formal distribution Lie superalgebra} of $S$ (see \cite{K3}).

\subsection{Lie   superalgebra}

  Let $V=V_{\bar0}\oplus V_{\bar1}$ be a $\Z_2$-graded vector space. Then any element $v\in V_{\bar0}$
is said to be \textit{even}  and any element $v\in V_{\bar1}$ is said to be \textit{odd}. Define $|v|=0$ if
$v$ is even and $|v|=1$ if $v$ is odd. Elements in  $V_{\bar0}$ or  $V_{\bar1}$ are called \textit{homogeneous}.
 Throughout the present paper, all elements in superalgebras and modules are homogenous unless
specified.

Assume that $\mathcal{G}$ is a Lie superalgebra. A \textit{$\mathcal{G}$-module} is a $\Z_2$-graded vector space $V$ together
with a bilinear map $\mathcal{G}\times V\rightarrow V$, denoted $(x,v)\mapsto xv$ such that
$$x(yv)-(-1)^{|x||y|}y(xv)=[x,y]v$$
and
$$\mathcal{G}_{\bar i} V_{\bar j}\subseteq  V_{\bar i+\bar j}$$
for all $x, y \in \mathcal{G}, v \in V $. Thus there is a parity-change functor $\Pi$ on the category of
$\mathcal{G}$-modules to itself. In other words, for any module
$V=V_{\bar0}\oplus V_{\bar1}$, we have a new module $\Pi(V)$ with the same underlining space with the
parity exchanged, i.e., $\Pi(V_{\bar 0})=V_{\bar 1}$ and $\Pi(V_{\bar 1})=V_{\bar0 }$. We use $U(\mathcal{G})$ to denote
the universal enveloping algebra.
All  modules   considered in this paper are $\Z_2$-graded and all  irreducible modules are
  non-trivial.

\begin{defi}\rm
Let $V$ be a module of a Lie superalgebra $\G$ and $x\in\G$.

{\rm (1)} If for any $v\in V$ there exists $n\in\Z_+$ such that $x^nv=0$,  then we call that the action of  $x$  on $V$ is  \textit{ locally nilpotent}.
Similarly, the action of $\G$   on $V$  is  \textit{ locally nilpotent} if for any $v\in V$ there exists $n\in\Z_+$ such that $\G^nv=0$.

{\rm (2)} If for any $v\in V$ we have $\mathrm{dim}(\sum_{n\in\Z_+}\C x^nv)<+\infty$, then we call that the action of $x$ on $V$ is   \textit{ locally finite}.
Similarly, the action of $\G$  on $V$  is   \textit{ locally finite}
 if for any $v\in V$ we have $\mathrm{dim}(\sum_{n\in\Z_+} {\G}^nv)<+\infty$.
\end{defi}
We note that   the action of $x$ on $V$ is locally nilpotent implies that the action of $x$ on $V$ is locally finite.
If $\G$ is a finitely generated Lie superalgebra,  the action of $\G$ on $V$ is locally nilpotent implies that the action of
 $\G$
 on $V$ is locally finite.

\section{The super Heisenberg-Virasoro algebra}
\subsection{The Lie superalgebra  $\mathfrak{S}$}
 \textit{The Heisenberg-Virasoro  Lie conformal algebra} $\mathfrak{v}$
is a free $\C[\partial]$-module generated by $L$ and  $I$   satisfying
 $$[L\, {}_\lambda \, L]=(\partial+2\lambda)L,\
 [L\, {}_\lambda \, I]=(\partial+\lambda)I,\ [I\, {}_\lambda \, I]=0.$$
  It is
well known  that
 all non-trivial free  $\mathfrak{v}$-modules   of rank one over $\C[\partial]$ are defined as follows (see \cite{XY}):
$$V(a,b,c) = \C[\partial]v,\ L\,{}_\lambda\, v=(\partial + a\lambda+b)v,\ I\,{}_\lambda\, v=cv,$$
where $a,b,c\in\C$.
The module $V(a,b,c)$ is irreducible if and only if $(a,c)\neq(0,0)$.
 It follows from this that we consider a $\Z_2$-graded $\C[\partial]$-module
$$\mathcal{S}(\phi,\varphi,a,b,c)=\mathcal{S}_{\bar 0}\oplus\mathcal{S}_{\bar 1}$$
with
$\mathcal{S}_{\bar 0}=\C[\partial]L\oplus\C[\partial]I,
\mathcal{S}_{\bar 1}=\C[\partial]G$ and satisfying
\begin{eqnarray}\label{3.11}
&& [L\, {}_\lambda \, G]=(\partial + a\lambda+b) G,
\  [I\, {}_\lambda \, G]=cG,
\ [G\, {}_\lambda \, G]=\phi(\partial,\lambda)L+\varphi(\partial,\lambda)I,
\end{eqnarray}
where  $\phi(\partial,\lambda),\varphi(\partial,\lambda)\in\C[\partial,\lambda]$ without $\phi(\partial,\lambda)=\varphi(\partial,\lambda)=0$.
\begin{lemm}\label{le511}
Let   $a,b,c\in\C$.   Then the  $\Z_2$-graded $\C[\partial]$-module  $\mathcal{S}(\phi,\varphi,a,b,c)$ becomes   a Lie conformal superalgebra
if and only if  $a=1$, $\phi(\partial,\lambda)=b=c=0$, $\varphi(\partial,\lambda)=\Delta\in\C^*$.
\end{lemm}
\begin{proof}
First we prove the necessity.
   Assume that $\mathcal{S}(\phi,\varphi,a,b,c)$ is a Lie conformal superalgebra.
 Using the Jacobi identity for triple $(I, G, G)$,  we have
 \begin{eqnarray}\label{324.11}
&& c(\phi(\partial,\lambda+\mu)+\phi(\partial,\mu))=0,
\\&&\label{324.22}   \lambda\phi(\partial+\lambda,\mu)=c(\varphi(\partial,\lambda+\mu)+\varphi(\partial,\mu)).
\end{eqnarray}
Letting $\lambda=0$ in \eqref{324.11} and \eqref{324.22},  we have $2c\phi(\partial,\mu)=2c\varphi(\partial,\mu)=0$.
 If $c\neq0$,  one can see that $\phi(\partial,\mu)=\varphi(\partial,\mu)=0$.    If $c=0$, by \eqref{324.22}, one has $\phi(\partial,\lambda)=0$.
 Then the third relations of \eqref{3.11} can be rewritten as $[G\, {}_\lambda \, G]=\varphi(\partial,\lambda)I$.
By the Jacobi identity for triple $(L, G, G)$, we get
 \begin{eqnarray}
 \label{555.66}(\partial+\lambda)\varphi(\partial+\lambda,\mu)
  =(-\mu+(a-1)\lambda+b)\varphi(\partial,\lambda+\mu)
  +(\partial+\lambda a+\mu+b)\varphi(\partial,\mu).
\end{eqnarray}
If $b\neq0$, putting  $\lambda=0$ into   \eqref{555.66}, we have $\varphi(\partial,\mu)=0$. This contradicts the hypothesis.
Consider $b=0$.
 Setting  $\partial=0$ in   \eqref{555.66},
 we have
 \begin{eqnarray}\label{77.88}
 &&\mu\frac{\varphi(0,\mu+\lambda)-\varphi(0,\mu)}{\lambda}
  =a\varphi(0,\mu)+(a-1)\varphi(0,\lambda+\mu)
  -\varphi(\lambda,\mu).
\end{eqnarray}
Taking $\lambda\rightarrow0$, we have $\mu\frac{d}{d\mu}\varphi(0,\mu)=2(a-1)\varphi(0,\mu)$. If $2(a-1)\notin\Z_+$, one has  $\varphi(0,\mu)=0$.
Inserting this into \eqref{77.88}, we get  $\varphi(\partial,\lambda)=0$. Then consider $2(a-1)\in\Z_+$, we have $\varphi(0,\mu)=\Delta \mu^{2(a-1)}$ for $\Delta\in\C^*$.
 Putting this into \eqref{77.88},  we obtain
  \begin{eqnarray}\label{8888.88}
  \varphi(\lambda,\mu)= a\Delta \mu^{2(a-1)}+(a-1)\Delta {(\lambda+\mu)}^{2(a-1)}-\mu\Delta\frac{(\lambda+\mu)^{2(a-1)}-\mu^{2(a-1)}}{\lambda}
  \end{eqnarray}
Letting $\partial=-\lambda$ in \eqref{555.66}, one can see that
   \begin{eqnarray}\label{9087.88}((a-1)\lambda-\mu)\varphi(-\lambda,\lambda+\mu)
  +((a-1)\lambda+\mu)\varphi(-\lambda,\mu)=0.
  \end{eqnarray} Then  using \eqref{8888.88} in \eqref{9087.88} and comparing the highest degree of $\lambda$, we check that $a=1$.
  Hence,  $\varphi(\partial,\lambda)=\Delta\in\C^*$.

In a word, by the definition of  Lie conformal superalgebra,  we get $a=1$, $b=c=0$, $\phi(\partial,\lambda)=0$ and $\varphi(\partial,\lambda)=\Delta\in\C^*$.

Based on Definition \ref{D1}, the sufficiency is clear.
\end{proof}

 Up to isomorphism, we may assume that $\Delta=2$
  in  Lemma \ref{le511}.  We can define a   finite Lie conformal superalgebra called  \textit{the Heisenberg-Virasoro Lie conformal superalgebra } $\mathfrak{s} =\mathfrak{s}_{\bar 0}\oplus\mathfrak{s}_{\bar 1}$
with $\mathfrak{s}_{\bar 0}=\C[\partial]L\oplus\C[\partial]I,\mathfrak{s}_{\bar 1}=\C[\partial]G$, which
  satisfies the following non-trivial $\lambda$-brackets
 \begin{eqnarray*}
&&[L\, {}_\lambda \, L]=(\partial+2\lambda) L,
\ [L\, {}_\lambda \, I]=(\partial+\lambda) I,
\ [L\, {}_\lambda \, G]=(\partial+\lambda) G,
\   [G\, {}_\lambda \, G]=2I.
\end{eqnarray*}

 In the rest of this section,  an infinite dimensional  Lie superalgebra  associated with the twisted Heisenberg-Virasoro algebra  is presented.
\begin{lemm}\label{lemma3225}
A formal distribution Lie superalgebra  of $\mathfrak{s}$ is given by
$$\mathrm{Lie}(\mathfrak{s})=\Big\{L_{m},I_{m},G_{m}\mid m\in \Z\Big\}$$
with    non-vanishing    relations:

\begin{eqnarray*}
&&[L_{m},L_{n}]=(n-m)L_{m+n},\
 [L_{m},I_{n}]=nI_{m+n},
 \\&&
 [L_{m},G_{n}]=nG_{m+n},\
   [G_{m},G_{n}]=2I_{m+n},
\end{eqnarray*}
where $m,n\in\Z$.
\end{lemm}
\begin{proof}
By the definition of the $k$-th product and   $\mathfrak{s}$,    we conclude that
\begin{eqnarray*}
&&L\,{}_{{}_{(0)}} L=\partial L, \ L\,{}_{{}_{(1)}} L=2 L,\  L\,{}_{{}_{(0)}} I=\partial I,\  L\,{}_{{}_{(1)}} I=I,
\\&& L\,{}_{{}_{(0)}} G=\partial G, \ L\,{}_{{}_{(1)}} G= G,\  G\,{}_{{}_{(0)}} G=2 I,
\\&& L\,{}_{{}_{(i)}} L= L\,{}_{{}_{(i)}} I=L\,{}_{{}_{(i)}} G=G\,{}_{{}_{(i-1)}} G=I\,{}_{{}_{(i-2)}} I= G\,{}_{{}_{(i-2)}} I=0,\ \forall \ i\geq2.
\end{eqnarray*}
From  \eqref{2011}, it is easy to check  that
 \begin{equation}\label{44.2}
\aligned
&[L_{(m)},L_{(n)}]=(m-n)L_{(m+n-1)},
\ [L_{(m)},I_{(n)}]=nI_{(m+n-1)},
\\&[L_{(m)},G_{(n)}]=nG_{(m+n-1)},
\ [G_{(m)},G_{(n)}]=2I_{(m+n)}.
\endaligned
\end{equation}
Then the lemma is proved by  making the shift
$$L_{m}\rightarrow L_{(m+1)},\ I_{m}\rightarrow I_{(m)},\ G_{m}\rightarrow G_{(m)}$$ in \eqref{44.2} where  $m\in\Z$.
\end{proof}
Note that  the
formal distribution Lie superalgebra $\text{Lie}(\mathfrak{s})$  is   exactly the super Heisenberg-Virasoro algebra  of
Ramond type.
\subsection{The super Heisenberg-Virasoro algebra   of  Neveu-Schwarz type}
 \textit{ The  super Heisenberg-Virasoro algebra  of  Neveu-Schwarz type} $\mathcal{T}$
 is defined as an infinite dimensional Lie superalgebra over $\C$ with basis
$\{L_m,I_m,G_{r}
\mid m\in \Z, r\in \frac{1}{2}+\Z\}$   satisfying the following  non-trivial relations
\begin{eqnarray*}
&&[L_{m},L_{n}]=(n-m)L_{m+n},
\
 [L_{m},I_{n}]=nI_{m+n},
 \\&&
  [L_{m},G_{r}]=rG_{m+r},
   \
 [G_{r},G_{s}]=2I_{r+s}
\end{eqnarray*} for $m,n\in\Z,r,s\in\frac{1}{2}+\Z$. Note that $\mathcal{T}=\mathcal{T}_{\bar0}\oplus\mathcal{T}_{\bar1}$, where $\mathcal{T}_{\bar0}=\{L_m,I_m\mid m\in\Z\}$ and  $\mathcal{T}_{\bar1}=\{G_r\mid r\in\frac{1}{2}+\Z\}$.

We know that  ${\mathcal{T}}$  is a subalgebra of   $\S$.
Let $\psi:\mathcal{T}\rightarrow \S$ be a linear map defined by
 \begin{eqnarray*}
  L_m&\mapsto&\frac{1}{2}L_{2m}
 \\I_m&\mapsto&I_{2m}
 \\G_{m+\frac{1}{2}}&\mapsto& G_{2m+1}
 \end{eqnarray*}
 for $m\in\Z$.
 It is straightforward to verify that $\psi$ is injective.

\section{Construction of simple restricted  $\S$-modules}
In this section, we will study a class of simple  restricted modules over the super Heisenberg-Virasoro algebra of
Ramond type.
\subsection{Some notations and definitions}
We denote by  $\mathbf{M}$  the set of all infinite vectors of the form $\mathbf{i}:=(\ldots, i_2, i_1)$ with entries in $\Z_+$,
satisfying the condition that the number of nonzero entries is finite and $\widehat{\mathbf{M}}=\{\mathbf{i}\in\mathbf{M}\mid i_k=0,1,\ \forall k\in\N\}$.
Write $\mathbf{0}=(\ldots, 0, 0)\in\mathbf{M}$ (or $\widehat{\mathbf{M}}$) and $\mathbb{Y}=\{0,1\}$.
For
$i\in\N$, denote  $\epsilon_i=(\ldots,0,1,0,\ldots,0)\in\mathbf{M}$ (or $\widehat{\mathbf{M}}$),
where $1$ is
in the $i$'th  position from the right. For any $\mathbf{i}\in\mathbf{M}$ (or $\widehat{\mathbf{M}}$), we write
$$\mathbf{w}(\mathbf{i})=\sum_{s\in\N}s\cdot i_s,$$
which is non-negative integer. For any    $\mathbf{0}\neq\mathbf{i}\in\mathbf{M}$ (or $\widehat{\mathbf{M}}$), assume that $p$  is the smallest integer such that $i_p\neq0$,
and define  $\mathbf{i}^\prime=\mathbf{i}-\epsilon_p$.

\begin{defi}\rm\label{defin421}
 Denote by $\succ$ the   \textit{reverse  lexicographical  total order}  on  $\mathbf{M}$ (or $\widehat{\mathbf{M}}$),  defined as follows: $\mathbf{0}$ is the minimum element;
  for different nonzero  $\mathbf{i},\mathbf{j}\in\mathbf{M}$ (or $\widehat{\mathbf{M}}$)
$$\mathbf{j} \succ \mathbf{i} \ \Longleftrightarrow
 \ \mathrm{ there\ exists} \ r\in\Z_+ \ \mathrm{such \ that} \ (j_s=i_s,\ \forall 1\leq s<r) \ \mathrm{and} \ j_r>i_r.$$
\end{defi}
Clearly, we can restrict the reverse lexicographic total order on $\mathbf{M}$ (resp. $\widehat{\mathbf{M}}$) to $\Z^n_+$ (resp. $\mathbb{Y}^n$) for $n\in \mathbb{N}$.
Then we can induce a \textit{principal total order} on $\mathbf{M}\times   \widehat{\mathbf{M}}  \times\mathbf{M}$, still denoted by $\succ$: for   $\mathbf{i},\mathbf{k},\mathbf{l},\mathbf{n}\in\mathbf{M}$, $\mathbf{j},\mathbf{m}\in\widehat{\mathbf{M}}$ and    $(\mathbf{i},\mathbf{j},\mathbf{k})\neq(\mathbf{l},\mathbf{m},\mathbf{n})$, set
\begin{eqnarray*}
&&(\mathbf{i},\mathbf{j},\mathbf{k}) \succ (\mathbf{l},\mathbf{m},\mathbf{n})
\\&&\Longleftrightarrow
\big(\mathbf{i},\mathbf{w}(\mathbf{i}),\mathbf{j},\mathbf{w}(\mathbf{j}),\mathbf{k},\mathbf{w}(\mathbf{k})\big) \succ
 \big(\mathbf{l},\mathbf{w}(\mathbf{l}),\mathbf{m},\mathbf{w}(\mathbf{m}),\mathbf{n},\mathbf{w}(\mathbf{n})\big).
\end{eqnarray*}
For any $\alpha,\beta,r,s,t\in\Z_+$ with $\alpha\geq2\beta$,   set
\begin{eqnarray*}
&&\S_{\alpha,\beta}=\bigoplus_{i\geq0}\C L_i\oplus\bigoplus_{i\geq0}\C I_{i-\alpha}\oplus\bigoplus_{i\geq0}\C G_{i-\beta},
\\&&\S^{(r,s,t)}=\bigoplus_{i\geq r}\C L_i\oplus\bigoplus_{i\geq s}
\C I_{i}\oplus\bigoplus_{i\geq t}\C G_{i}.
\end{eqnarray*}
Let  $V$ be a simple $\S_{\alpha,\beta}$-module. Then  we have the induced   $\S$-module
$$\mathrm{Ind}(V)=U(\S)\otimes_{U(\S_{\alpha,\beta})}V.$$
Considering  simple modules of the superalgebra $\S$ or one of its subalgebras   containing the central elements $I_0$,
  we always assume that the action  of $I_0$ is a scalar $c_0$.

  For   $\mathbf{i}$, $\mathbf{k}\in \mathbf{M}$,    $\mathbf{j}\in\widehat{\mathbf{M}}$,  we denote
$$I^{\mathbf{i}} G^{\mathbf{j}} L^{\mathbf{k}}=\cdots I_{-2-\alpha}^{i_2} I_{-1-\alpha}^{i_1}\cdots G_{-{2}-\beta}^{j_2} G_{-1-\beta}^{j_1}\cdots L_{-2}^{k_2} L_{-1}^{k_1}\in U(\S).$$
From the $\mathrm{PBW}$ Theorem (see \cite{CW}) and $G^2_{i}=I_{2i}$ for $i\in\Z$, every element of $\mathrm{Ind}(V)$ can be uniquely written as
follow
\begin{equation}\label{def2.1}
\sum_{\mathbf{i},\mathbf{k}\in\mathbf{M}, \mathbf{j}\in\widehat{\mathbf{M}}}I^{\mathbf{i}} G^{\mathbf{j}} L^{\mathbf{k}} v_{\mathbf{i},\mathbf{j},\mathbf{k}},
\end{equation}
where all  $v_{\mathbf{i},\mathbf{j},\mathbf{k}}\in V$ and only finitely many of them are nonzero. For any $v\in\mathrm{Ind}(V)$ as in  \eqref{def2.1}, we denote by $\mathrm{supp}(v)$ the set of all $(\mathbf{i},\mathbf{j},\mathbf{k})\in \mathbf{M}\times \widehat{\mathbf{M}}\times\mathbf{M}$  such that $v_{\mathbf{i},\mathbf{j},\mathbf{k}}\neq0$.
 For a nonzero $v\in \mathrm{Ind}(V)$,
 we write $\mathrm{deg}(v)$  the maximal element in $\mathrm{supp}(v)$ by the principal total order on $\mathbf{M}\times\widehat{\mathbf{M}}\times\mathbf{M}$,
  which is called the \textit{degree} of $v$. Note that $\mathrm{deg}(v)$ is defined only for  $v\neq0$.
\begin{defi}\rm
An $\S$-module $W$ is called \textit{restricted} if for any $w\in W$ there exists
$k\in\Z_+$   such that $L_iw=G_iw=I_iw=0$ for   $i>k$.
\end{defi}

\subsection{Simple restricted  $\S$-modules}
The purpose of this section is to obtain  some simple restricted $\S$-modules.
\begin{theo}\label{th1}
Let $V$ be a simple $\S_{\alpha,\beta}$-module. Assume that there exists $z\in\Z_+$  such that $V$ satisfies  the following two conditions:
\begin{itemize}
\item[{\rm (a)}] the action of  $I_{z}$ on   $V$  is injective,

\item[{\rm (b)}]
   $L_iV=G_jV=I_kV=0$ for  all $i>z+\alpha,j>z+\beta$ and $k>z$.
\end{itemize}
 Then we obtain that
 $\mathrm{Ind}(V)$ is a simple $\S$-module.
\end{theo}

\begin{proof}
 For any $v\in\ \mathrm{Ind}(V)\setminus V$, suppose that $v$ as the form  \eqref{def2.1}, i.e.,
$$v=\sum_{\mathbf{x},\mathbf{z}\in\mathbf{M}, \mathbf{y}\in\widehat{\mathbf{M}}}I^{\mathbf{x}} G^{\mathbf{y}} L^{\mathbf{z}} v_{\mathbf{x},\mathbf{y},\mathbf{z}}.$$
Denote  $$\mathrm{deg}(v)=(\mathbf{i},\mathbf{j},\mathbf{k})$$
  for $\mathbf{i},\mathbf{k}\in \mathbf{M},\mathbf{j}\in \widehat{\mathbf{M}}$,
and let
$\tilde{i}=\mathrm{min}\{s:i_s\neq0\}$ if $\mathbf{i}\neq\mathbf{0}$, ${\tilde{j}}=\mathrm{min}\{s:j_s\neq0\}$ if $\mathbf{j}\neq\mathbf{0}$,  $\tilde{k}=\mathrm{min}\{s:k_s\neq0\}$ if $\mathbf{k}\neq\mathbf{0}$.
To prove this theorem,  we have  three steps as follows.

First,  consider those $v_{\mathbf{x},\mathbf{y},\mathbf{z}}$
with $I_{\tilde{k}+z}I^{\mathbf{x}}G^{\mathbf{y}}L^{\mathbf{z}}v_{\mathbf{x},\mathbf{y},\mathbf{z}}\neq0.$
It is easy to see that
$$I_{\tilde{k}+z}I^{\mathbf{x}}G^{\mathbf{y}}L^{\mathbf{z}}v_{\mathbf{x},\mathbf{y},\mathbf{z}}=
I^{\mathbf{x}}G^{\mathbf{y}}[I_{\tilde{k}+z},L^{\mathbf{z}}]v_{\mathbf{x},\mathbf{y},\mathbf{z}}.$$
Clearly, $I_{z}v_{\mathbf{x},\mathbf{y},\mathbf{z}}\neq0$. Consider the following two cases.

If $\mathbf{z}=\mathbf{k}$, one can get
$$\mathrm{deg}(I_{\tilde{k}+z}I^{\mathbf{x}}G^{\mathbf{y}}L^{\mathbf{z}}v_{\mathbf{x},\mathbf{y},\mathbf{z}})=
(\mathbf{x},\mathbf{y},\mathbf{k}^{\prime})\preceq (\mathbf{i},\mathbf{j},\mathbf{k}^{\prime}),$$
where the equality holds if and only if $\mathbf{x}=\mathbf{i},\mathbf{y}=\mathbf{j}$.

Suppose  $(\mathbf{x},\mathbf{w}(\mathbf{x}),\mathbf{y},\mathbf{w}(\mathbf{y}),\mathbf{z},\mathbf{w}(\mathbf{z}))
\prec(\mathbf{i},\mathbf{w}(\mathbf{i}),\mathbf{j},\mathbf{w}(\mathbf{j}),\mathbf{k},\mathbf{w}(\mathbf{k}))$ and denote  $$\mathrm{deg}(I_{\tilde{k}+z}I^{\mathbf{x}}G^{\mathbf{y}}L^{\mathbf{z}}v_{\mathbf{x},\mathbf{y},\mathbf{z}})
=(\mathbf{x}_1,\mathbf{y}_1,\mathbf{z}_1)\in \mathbf{M}\times \widehat{\mathbf{M}}\times \mathbf{M}.$$ If $\mathbf{w}(\mathbf{z})<\mathbf{w}(\mathbf{k})$,
then we obtain $\mathbf{w}(\mathbf{z}_1)\leq \mathbf{w}(\mathbf{z})-\tilde{k}<\mathbf{w}(\mathbf{k})-\tilde{k}=\mathbf{w}(\mathbf{k}^{\prime})$,
which implies $(\mathbf{x}_1,\mathbf{y}_1,\mathbf{z}_1)\prec(\mathbf{i},\mathbf{j},\mathbf{k}^{\prime})$.
Then  suppose $\mathbf{w}(\mathbf{z})=\mathbf{w}(\mathbf{k})$ and ${\mathbf{z}}\prec{\mathbf{k}}$. Let $\tilde{z}:=\mathrm{min}\{s:z_s\neq0\}>0$.
If $\tilde{z}>\tilde{k}$, one can see that $\mathbf{w}(\mathbf{z}_1)=\mathbf{w}(\mathbf{z})-\tilde{z}<\mathbf{w}(\mathbf{z})-\tilde{k}=\mathbf{w}(\mathbf{k}^{\prime})$. If $\tilde{z}=\tilde{k}$,
by the similar method, we deduce that $(\mathbf{x}_1,\mathbf{y}_1,\mathbf{z}_1)=(\mathbf{x},\mathbf{y},\mathbf{z}^{\prime})$.
Then by  $\mathbf{z}^{\prime}\prec \mathbf{k}^{\prime}$, one can see that $\mathrm{deg}(I_{\tilde{k}+z}I^{\mathbf{x}}G^{\mathbf{y}}L^{\mathbf{z}}v_{\mathbf{x},\mathbf{y},\mathbf{z}})
=(\mathbf{x}_1,\mathbf{y}_1,\mathbf{z}_1)\prec(\mathbf{i},\mathbf{j},\mathbf{k}^{\prime})$ in both cases.

Combining all the arguments above we conclude that
$\mathrm{deg}(I_{\tilde{k}+z}v)=(\mathbf{i},\mathbf{j},\mathbf{k}^{\prime})$.

  Now  we consider   $v_{\mathbf{x},\mathbf{y},\mathbf{0}}$
with $G_{{\tilde{j}}+z}I^{\mathbf{x}}G^{\mathbf{y}}v_{\mathbf{x},\mathbf{y},\mathbf{0}}\neq0.$
According to $I_{z}v_{\mathbf{x},\mathbf{y},\mathbf{0}}\neq0$ for any  $(\mathbf{x},\mathbf{y},\mathbf{0})\in \mathrm{supp}(v)$,
we check that
$$G_{{\tilde{j}}+z}I^{\mathbf{x}}G^{\mathbf{y}}v_{\mathbf{x},\mathbf{y},\mathbf{0}}=
I^{\mathbf{x}}[G_{{\tilde{j}}+z},G^{\mathbf{y}}]v_{\mathbf{x},\mathbf{y},\mathbf{0}}.$$
If $\mathbf{y}=\mathbf{j}$, one can get  that
$$\mathrm{deg}(G_{{\tilde{j}}+z}I^{\mathbf{x}}G^{\mathbf{y}}v_{\mathbf{x},\mathbf{y},\mathbf{0}})=
(\mathbf{x},\mathbf{y}^{\prime},\mathbf{0})\preceq (\mathbf{i},\mathbf{j}^{\prime},\mathbf{0}),$$
where the equality holds if and only if $\mathbf{x}=\mathbf{i}$.

Suppose $(\mathbf{x},\mathbf{w}(\mathbf{x}),\mathbf{y},\mathbf{w}(\mathbf{y}),\mathbf{0},\mathbf{0})
\prec(\mathbf{i},\mathbf{w}(\mathbf{i}),\mathbf{j},\mathbf{w}(\mathbf{j}),\mathbf{0},\mathbf{0})$ . Then   we have  $$\mathrm{deg}(G_{{\tilde{j}}+z}I^{\mathbf{x}}G^{\mathbf{y}}v_{\mathbf{x},\mathbf{y},\mathbf{0}})
=(\mathbf{x_1},\mathbf{y_1},\mathbf{0})\in \mathbf{M}\times \widehat{\mathbf{M}}\times \mathbf{M}.$$ If $\mathbf{w}(\mathbf{y})<\mathbf{w}(\mathbf{j})$,
then we have $\mathbf{w}(\mathbf{y}_1)\leq \mathbf{w}(\mathbf{y})-{\tilde{j}}<\mathbf{w}(\mathbf{j})-{\tilde{j}}=\mathbf{w}(\mathbf{j}^{\prime})$,
which gives $(\mathbf{x}_1,\mathbf{y}_1,\mathbf{0})\prec(\mathbf{i},\mathbf{j}^{\prime},\mathbf{0})$.

Then we suppose $\mathbf{w}(\mathbf{y})=\mathbf{w}(\mathbf{j})$ and ${\mathbf{y}}\prec{\mathbf{j}}$. Let $\tilde{y}:=\mathrm{min}\{s:y_s\neq0\}>0$.
If $\tilde{y}>{\tilde{j}}$, we obtain $\mathbf{w}(\mathbf{y}_1)=\mathbf{w}(\mathbf{y})-{\tilde{y}}<\mathbf{w}(\mathbf{j})-{\tilde{j}}=\mathbf{w}(\mathbf{j}^{\prime})$. If $\tilde{y}={\tilde{j}}$,
we can similarly check that $(\mathbf{x}_1,\mathbf{y}_1,\mathbf{0})=(\mathbf{x},\mathbf{y}^{\prime},\mathbf{0})$.
Now from $\mathbf{y}^{\prime}\prec \mathbf{j}^{\prime}$, we have $\mathrm{deg}(G_{{\tilde{j}}+z}I^{\mathbf{x}}G^{\mathbf{y}}v_{\mathbf{x},\mathbf{y},\mathbf{0}})
=(\mathbf{x}_1,\mathbf{y}_1,\mathbf{0})\prec(\mathbf{i},\mathbf{j}^{\prime},\mathbf{0})$.
Therefore, we conclude that
$\mathrm{deg}(G_{{\tilde{j}}+z}I^{\mathbf{x}}G^{\mathbf{y}}v_{\mathbf{x},\mathbf{y},\mathbf{0}})=(\mathbf{i},\mathbf{j}^{\prime},\mathbf{0})$.

 At last, we  consider   $v_{\mathbf{x},\mathbf{0},\mathbf{0}}$
with  $L_{{\tilde{i}}+z}I^{\mathbf{x}}v_{\mathbf{x},\mathbf{0},\mathbf{0}}\neq0.$
By the similar arguments as above, we have $\mathrm{deg}(L_{{\tilde{i}}+z}I^{\mathbf{x}}v_{\mathbf{x},\mathbf{0},\mathbf{0}})=(\mathbf{i}^{\prime},\mathbf{0},\mathbf{0})$.

Thus, from any  $0\neq v\in\mathrm{Ind}(V)$ we can reach a nonzero element in
 $U(\S )v\cap V\neq0$, which gives the simplicity of $\mathrm{Ind}(V)$.
We complete  the proof.
\end{proof}
\begin{rema}\label{rema3.3} \rm
We note that  the induced modules $\mathrm{Ind}(V)$ are simple restricted $\S$-modules.
 \end{rema}



\section{Characterization of simple $\S$-modules}
\begin{prop}\label{th2}
Let ${P}$ be a simple  $\S$-module. Suppose that there exists $k\in\Z_+$ such that the action of $I_k$ on $P$ is injective. Then the following conditions are equivalent:

\begin{itemize}
\item[{\rm (1)}]  There exists $z\in\Z_+$ such that the actions of $L_i,G_{i},I_i,i\geq z$ on $P$ are locally finite.

\item[{\rm (2)}]   There exists $z\in\Z_+$ such that the actions of $L_i,G_{i},I_i,i\geq z$ on $P$ are locally nilpotent.

\item[{\rm (3)}]  There exists $r,s,t\in\Z_+$ such that  $P$ is a  locally finite $\S^{(r,s,t)}$-module.

\item[{\rm (4)}]    There exists $r,s,t\in\Z_+$ such that  $P$ is a  locally nilpotent $\S^{(r,s,t)}$-module.

\item[{\rm (5)}]  There exists    a simple $\S_{\alpha,\beta}$-module $V$ satisfying the conditions in Theorem \ref{th1}  such that
$P\cong\mathrm{Ind}(V)$.
\end{itemize}
\end{prop}

\begin{proof}
 Note that $(5)\Rightarrow(3)\Rightarrow(1)$, $(5)\Rightarrow(4)\Rightarrow(2)$
 and $(2)\Rightarrow(1)$ are easy to get.

Now we    prove $(1)\Rightarrow(5)$.
Assume that $P$ is a simple $\S$-module and  there exists $z\in\Z_+$ such that the actions of $L_i,G_{i},I_i,i\geq z$ on $P$ are locally finite. Then we can choose a $0\neq v\in P$ such that $L_zv=\gamma v$ for some $\gamma\in \C$.

In the following, we always assume that  $X\in\{L,I,G\}$ for simplicity. Choose   $j\in\Z_+$ with $j> z$ and we  denote
 \begin{eqnarray*}
 &&S_X=\sum_{m\in\Z_+}\C L_{z}^mX_{j}v={U}(\C L_{z})X_{j}v,
 \end{eqnarray*}
 which are all finite dimensional. From the definition of $\S$ and any $m\in\Z_+$,  we check that
\begin{eqnarray*}
&&X_{j+mz}v\in S_X\Rightarrow X_{j+(m+1)z}v\in S_X,
\end{eqnarray*}
where  $j>z$.
According to  induction on $m\in\Z_+$, we   have
 $X_{j+mz}v\in S_X$.
As a matter of fact,  we  know that  $\sum_{m\in\Z_+}\C X_{j+mz}v$  are
finite dimensional for  $j> z$. Hence,
 \begin{eqnarray*}
&&\sum_{i\in\Z_+}\C X_{z+i}v=\C X_{z}v+\sum_{j=z+1}^{2z}\big(\sum_{m\in\Z_+}\C X_{j+mz}v\big)
\end{eqnarray*}
 are all finite dimensional. Now    we can choose $l\in\Z_+$ such that
\begin{eqnarray*}
\sum_{i\in\Z_+}\C X_{z+i}v=\sum_{i=0}^{l}\C X_{z+i}v.
\end{eqnarray*}
Then denote  $V^\prime=\sum_{x_0,\ldots,x_l,y_0,\ldots,y_l\in\N,z_0,\ldots,z_l\in\mathbb{Y}}\C L_{z}^{x_0}\cdots
  L_{z+l}^{x_l}I_{z}^{y_0}\cdots
  I_{z+l}^{y_l}G_{z}^{z_0}\cdots
  G_{z+l}^{z_l}v$.
It is easy to see that
 $V^\prime$ is a (finite dimensional) $\S^{(z,z,z)}$-module.

It follows that we can  take
 a minimal $n\in\Z_+$ such that
 \begin{equation}\label{lm3.3}
 (L_m+a_1L_{m+1}+\cdots + a_{n}L_{m+{n}})V^\prime=0
 \end{equation}
 for some $m> z$ and  $a_i\in \C,i=1,\ldots,n$.
 Applying $L_m$ to \eqref{lm3.3}, one gets
$$(a_1[L_m,L_{m+1}]+\cdots +a_{n}[L_m,L_{m+n}])V^\prime=0,$$
which shows $n=0$, that is, $L_mV^\prime=0$ for $m> z$.
Therefore, for $m,i> z$, we have
$$0=G_{i}L_{m}V^\prime=[G_{i},L_{m}]V^\prime+L_m(G_{i}V^\prime)=-i G_{m+ i}V^\prime,$$
namely, $G_{m+i}V^\prime=0$. Similarly, we have  $I_{m+i} V^\prime=L_{m+i} V^\prime=0$ for all $i> z$.

 For any $a^\prime,b^\prime,c^\prime\in \Z$, we consider the vector space
 $$\V_{a^\prime,b^\prime,c^\prime}=\{v\in P\mid L_iv=G_{j}v=I_kv=0 \quad \mbox{for\ any}\ i>a^\prime,j>b^\prime,k>c^\prime\}.$$
Clearly, $\V_{a^\prime,b^\prime,c^\prime}\neq0$ for sufficiently large $a^\prime,b^\prime,c^\prime\in\Z_+$.
 On the other hand, $\V_{a^\prime,b^\prime,c^\prime}=0$ for all $c^\prime<0$ since there exists $k\in\Z_+$ such that the action of  $I_k$ on $P$ is injective.
Thus we can find a smallest non-negative integer  $c$, and choose some  $a,b> c$ with $a+c>2b$ such that  $\V_{a,b,c}\neq 0$.  Denote  $\alpha=a-c, \beta=b-c$ and  $V:=\V_{a,b,c}$.
Take  $r_1>a,r_2>b,r_3>c$. For $s\geq0$,   it follows from $r_1+s-\beta>b$  and $r_2+s-\beta>c$ that we can easily check that
\begin{eqnarray*}
&&L_{r_1}(G_{s-\beta}v)=(s-\beta)G_{r_1+s-\beta}v=0,
\\&&
G_{r_2}(G_{s-\beta}v)=2I_{r_2+s-\beta}v=0,\ I_{r_3}(G_{s-\beta}v)=0.
\end{eqnarray*}
Then $G_{s-\beta}v\in V$ for
all $s\geq 0$.
Considering $d\geq0$,     we have
\begin{eqnarray*}
&&L_{r_1}(I_{d-\alpha}v)=(d-\alpha)I_{r_1+d-\alpha}v=0,\
G_{r_2}(I_{d-\alpha}v)=I_{r_3}(I_{d-\alpha}v)=0,
\end{eqnarray*}
which shows $I_{d-\alpha}v\in V$. Similarly, we can also obtain   $L_{d}v\in V$
for all $d\in \Z_+$. Therefore, $V$ is an $\S_{\alpha,\beta}$-module, and  defined in Theorem \ref{th1}.


For $c\in\Z_+$, by the definition of $V$, we get that the action of  $I_{c}$  on $V$ is injective. Note that if $c=0$ we have $c_0\neq0$. Since $P$ is simple
and generated by $V$,   there exists a canonical surjective map
$$\pi:\mathrm{ Ind}(V) \rightarrow P, \quad \pi(1\otimes v)=v,\quad \forall  v\in V.$$
Next, we only need to prove that $\pi$ is also injective, i.e., $\pi$ as the canonical map is bijective.  Denote $K=\mathrm{ker}(\pi)$. Obviously, $K\cap V=0$. If	$K\neq0$, we can choose $0\neq v\in K\setminus V$ such that $\mathrm{deg}(v)=(\mathbf{i},\mathbf{j},\mathbf{k})$ is minimal possible.
Note that $K$ is an $\S $-submodule of $\mathrm{Ind}(V)$.
By the similar   proof in Theorem \ref{th1}, we can create a new vector $u\in K$  with $\mathrm{deg}(u)\prec(\mathbf{i},\mathbf{j},\mathbf{k})$, which  leads to a contradiction. This forces $K=0$,
i.e., $P\cong \mathrm{Ind}(V)$. Then we see that $V$ is a simple   $\S$-module.
This completes the proof.
\end{proof}

Now we present the main results of this section.
\begin{theo}\label{th3}
Let  $W$ be a  simple restricted $\S$-module. Suppose that there exists $k\in\Z_+$ such that the action of $I_k$ on $W$ is injective.  Then   $W\cong\mathrm{Ind}(V)$, where $V$ is a simple $\widehat{{\S}}^{(\alpha,\beta,z)}$-module,
 and $\widehat{{\S}}^{(\alpha,\beta,z)}={\S}_{(\alpha,\beta)}/\S^{(z+\alpha+1,z+1,z+\beta+1)}$ is a quotient algebra for some $z\in\Z_+$.
\end{theo}
\begin{proof}
Assume $0\neq w\in W$. We see that  there exists $t\in\Z_+$ such that $L_iw=G_{i}w=I_iw=0$ for all $i\geq t$. Every element $\widehat{w}$ in $W$ can be uniquely written as follow
$$\widehat{w}=\sum_{\mathbf{i},\mathbf{k}\in\mathbf{M}, \mathbf{j}\in\widehat{\mathbf{M}}}I^{\mathbf{i}} G^{\mathbf{j}} L^{\mathbf{k}} w.$$
Thus, for $i\geq t$, there exists a sufficiently large  $k\in\Z_+$ such that
 $$L_i^kW=G_{i}^kW=I_i^kW=0.$$
 This shows that  there exists $t\in\Z_+$ such that
the actions of $L_i, I_i, G_i$ for all $i\geq t$ on $W$ are locally nilpotent.
Then based on Proposition \ref{th2}, we have the results.
 \end{proof}

\section{Examples}
For convenience,
 we write
 $T{(r)}=\S^{(r,r,r)}$ for $r\in\Z_+$.  
Note that $T{(0)}=\S_0\oplus \S_+.$

\subsection{Simple    induced $\S$-modules }
\begin{prop}\label{th321}
Let $s\in\Z_{>1}$  and  $V$ be a simple $\S^{(0,s,0)}$-module. Assume that there exists  $z\in\Z_+$ with $z\geq s$ such that
 $I_{z}$ on   $V$  is injective,
   $L_iV=I_jV=G_kV=0$ and $I_0V=c_0 V$ for  all $i>z,j>z,k>z$ and $c_0\in\C$.
Then we obtain that
\begin{itemize}
$\mathrm{Ind}^{T(0)}_{\S^{(0,s,0)}}(V)$ is a simple $T(0)$-module, $I_{z}$ acts injectively on $\mathrm{Ind}^{T(0)}_{\S^{(0,s,0)}}(V)$
and $$L_{i}(\mathrm{Ind}^{T(0)}_{\S^{(0,s,0)}}(V))=I_j(\mathrm{Ind}^{T(0)}_{\S^{(0,s,0)}}(V))
=G_{k}(\mathrm{Ind}^{T(0)}_{\S^{(0,s,0)}}(V))=0,$$
where $i>z,j>z,k>z$.
\end{itemize}
\end{prop}

\begin{proof}
 Let $\succ$ be the reverse lexicographic total order on $\Z_+^{s}$  (see Definition \ref{defin421}).
 For any  $v\in\mathrm{Ind}^{T(0)}_{\S^{(0,s,0)}}(V)$, by $\mathrm{PBW}$ Theorem we may write $v$ in the form
 $$\sum_{\mathbf{x}\in\Z_+^{s}}I^{\mathbf{x}} v_{\mathbf{x}},$$ where
$I^{\mathbf{x}}=I_{s-1}^{x_{s-1}}I_{s-2}^{x_{s-2}}\cdots I_1^{x_1}$.
Let $\mathrm{supp}(v)$ be the set of all $\mathbf{i}$  with $v_{\mathbf{i}}\neq0$,  and let $\mathrm{deg}(v)$ be the maximal element of $\mathrm{supp}(v)$ with respect to reverse lexicographic order on $\Z_+^{s}$. We can check that   $$L_iv=I_{j}v=G_{k}v=0$$ for $i>z,j>z,k>z$.

Now for any $v\in \mathrm{Ind}^{T(0)}_{\S^{(0,s,0)}}(V)\setminus V$, let $\mathrm{deg}(v)=\mathbf{i}$ and
 $\widehat{s}=\mathrm{min}\{s :x_s\neq0\}$. We only need to show that
$\mathrm{deg}(L_{z-\widehat{s}}I^{\mathbf{k}} v_{\mathbf{k}})=\mathbf{k}^\prime\preceq\mathbf{i}^\prime$.
Then we can get that $\mathrm{Ind}^{T(0)}_{\S^{(0,s,0)}}(V)$ is a simple $T(0)$-module.
\end{proof}
It is clear that the induced module $W=\mathrm{Ind}^{T(0)}_{\S^{(0,s,0)}}(V)$ satisfies  the conditions in  Theorem \ref{th1} with $\alpha=\beta=0$. Then we obtain     simple $\S$-modules $\mathrm{Ind}(W)$.
\subsection{Highest weight modules}
For $h,c_0\in\C$, let $V_h$ be a 1-dimensional vector space over $\C$
spanned by $v_h$, namely, $V_h=\C v_h$. Regard $V_h$ as an $\S_0$-module such that  $L_0v=hv, I_0v=c_0v$.
Then $V_h$ is a $T{(0)}$-module by setting $\S_+\cdot V_h=0$. The Verma module $\mathfrak{V}(h,c_0)$ over $\S$
can be defined by
$$\mathfrak{V}(h,c_0)=U(\S)\otimes_{U(\S_0\oplus\S_+)}V_h,$$
and has the highest weight $(h,c_0)$. It is straightforward to check that   $\mathfrak{V}(h,c_0)$  is a
simple $\S$-module if $c_0\neq0$.
These simple modules are exactly the highest weight modules in Theorem \ref{th1} for $\alpha=\beta=z=0$.
\subsection{Whittaker modules}
We  recall the definition of the  classical Whittaker modules.
\begin{defi}\rm
 Assume that $\phi:T(1)\rightarrow\C$ is a Lie superalgebra homomorphism.
For $c_0\in\C$, an $\S$-module $\mathcal{W}$ is called \textit{a Whittaker module of type $(\phi,c_0)$} if

\noindent
{\rm (1)}  $\mathcal{W}$  is generated by a homogeneous vector $u$,

\noindent
{\rm (2)}   $xu=\phi(x)u$ for any $x\in T(1)$,

\noindent
{\rm (3)}  $I_0u=c_0u$,

where $u$ is called \textit{a Whittaker vector of $\mathcal{W}$}.
\end{defi}
\begin{lemm}\label{le6.1}
Let $V$ be the finite dimensional simple $T(1)$-module. Then $\mathrm{dim}(V)=1$.
\end{lemm}
\begin{proof}
In fact,
  there exists $m\gg0$ such that $L_mV=0$.
For $n\in\N$ and $m\gg0$, it is easy to get
$$0=[L_m,I_n]V=nI_{m+n}V\ \mathrm{and} \ 0=[L_m,G_{n}]V=nG_{m+n}V.$$
Then  for sufficiently large $m\in\Z_+$, we know that $V$ can be viewed as a module of quotient  algebra $T(1)/ T{(m)}$.
By  Lemma 1.33 of \cite{CW}, we immediately obtain $\mathrm{dim}(V)=1$.
\end{proof}
Let $\phi$ be a Lie superalgebra homomorphism $\phi:T(1)\rightarrow\C$.
  Then $\phi(L_i)=\phi(I_j)=\phi(G_{j})=0$ for $i\geq3,j\geq2$.
   Let $\mathfrak{t}_\phi=\C v$ be  a 1-dimensional vector space with
$$xv=\phi(x)v,\ I_0v=c_0v$$  for all $x\in T(1).$
Clearly, if $\phi(I_1)\neq0$,   $\mathfrak{t}_{\phi}$ is a simple $T(1)$-module and  $\mathrm{dim}(\mathfrak{t}_{\phi})=1$.
Now we consider the induced module
$$V_\phi=U(T{(0)})\otimes_{U(T(1))} \mathfrak{t}_\phi=\C[L_0]v\oplus G_0\C[L_0]v.$$
It is easy to see that $V_\phi$ is a simple $T{(0)}$-module if $\phi(I_1)\neq0$.
Then the simple induced $\S$-modules $\mathrm{Ind}(V_\phi)$ in Theorem \ref{th1}  are so-called Whittaker modules.
\subsection{High order Whittaker  modules}
Let $\phi_k$ be a Lie superalgebra homomorphism $\phi_k:T{(k)}\rightarrow\C$ for some $k\in\Z_+$.
  Then $\phi_k(L_i)=\phi_k(I_j)=\phi_k(G_{j})=0$ for $i\geq2k+1,j\geq2k$.
Assume that  $\mathfrak{t}_{\phi_k}=\C v$ is  a 1-dimensional vector space with
$$xv=\phi(x)v,\ I_0v=c_0v,$$  for all $x\in T(k).$
If $\phi(I_{{2k-1}})\neq0$,  $\mathfrak{t}_{\phi_k}$ is a simple $T(k)$-module and  $\mathrm{dim}(\mathfrak{t}_{\phi_k})=1$.
Consider the induced module
$$V_{\phi_k}=U(T(0))\otimes_{U(T(k))} \mathfrak{t}_\phi.$$
Clearly, $V_{\phi_k}$ is a simple $T(k)$-module if $\phi(I_{{2k-1}})\neq0$  and  $\mathrm{dim}(V_{\phi_k})=1$.
Then the corresponding simple  $\S$-modules $\mathrm{Ind}(V_{\phi_k})$ in Theorem \ref{th1}  are exactly the high order Whittaker modules.

\section*{Data availability}
The data that support the findings of this study are available from the corresponding author upon reasonable request.

\section*{Acknowledgements}
This work was supported by the National Natural Science Foundation of China
(Nos.11801369, 11871421, 12171129), the Overseas Visiting Scholars Program  of Shanghai Lixin University of  Accounting and Finance (No. 2021161),  the Zhejiang Provincial Natural Science Foundation of China (No.LY20A010022)  and the Scientific Research Foundation of Hangzhou Normal University (No.2019QDL012).
The authors would like to
thank the referee  for extensive suggestions to improve the paper.

\bigskip

Haibo Chen
\vspace{2pt}

  School of  Statistics and Mathematics, Shanghai Lixin University of  Accounting and Finance,   Shanghai
201209, China

\vspace{2pt}
Hypo1025@163.com

\bigskip

Xiansheng Dai
\vspace{2pt}

 School of Mathematical Sciences, Guizhou Normal  University, Guiyang 550001, China

\vspace{2pt}
daisheng158@126.com
\bigskip

Yanyong Hong
\vspace{2pt}

Department of Mathematics, Hangzhou Normal University,
Hangzhou 311121,  China

\vspace{2pt}
yyhong@hznu.edu.cn

\bigskip

\end{document}